\documentclass[12pt]{amsart}
\usepackage[colorlinks,linkcolor=blue,citecolor=blue,urlcolor=blue, pdfcenterwindow, pdfstartview={XYZ null null 1.2}, pdffitwindow, pdfdisplaydoctitle=true]{hyperref}
\usepackage[english]{babel}
\usepackage{array}
\usepackage{threeparttable}
\usepackage{rotating}
\usepackage[section]{placeins}
\usepackage{amssymb} 
\usepackage{amsmath}
\usepackage{mathrsfs}
\usepackage{mathabx}
\usepackage{enumerate}
\usepackage{color}
\usepackage{relsize}
\usepackage{amsrefs}        
\usepackage[all]{xy}
\usepackage{comment}
\edef\next{%
 \noexpand\ProvidesPackage{xy}[\xydate\space Xy-pic version \xyversion]}\next

\begin{document}
\newtheorem{lemma}{Lemma}[section]
\newtheorem{lemm}[lemma]{Lemma}
\newtheorem{prop}[lemma]{Proposition}
\newtheorem{coro}[lemma]{Corollary}
\newtheorem{theo}[lemma]{Theorem}
\newtheorem{conj}[lemma]{Conjecture}
\newtheorem{prob}{Problem}
\newtheorem{ques}{Question}
\newtheorem{rema}[lemma]{Remark}
\newtheorem{rems}[lemma]{Remarks}
\newtheorem{defi}[lemma]{Definition}
\newtheorem{defis}[lemma]{Definitions}
\newtheorem{exam}[lemma]{Example}

\newcommand{\N}{\mathbf N}
\newcommand{\Z}{\mathbf Z}
\newcommand{\R}{\mathbf R}
\newcommand{\Q}{\mathbf Q}
\newcommand{\C}{\mathbf C}

\title{Braiding and asymptotic Schur's orthogonality}
\begin{abstract} Let $\pi:G\to U(\mathcal H)$ be a unitary representation of a locally compact group. The braiding operator
$F:\mathcal H\otimes\mathcal H\to \mathcal H\otimes\mathcal H$, which flips the components of the Hilbert tensor product $F(v\otimes w)=w\otimes v$, belongs to the von Neumann algebra $W^*((\pi\otimes\pi)(G\times G))$ if and only if $\pi$ is irreducible. Suppose $G$ is semisimple over a local field. If $G$ is non-compact with finite center, $P<G$ is a minimal parabolic, $\pi:G\to U(L^2(G/P))$ is the quasi-regular representation, then
\[
	\lim_{n\to\infty}\frac{1}{\int_{B_n}\Xi(g)^2dg}\int_{B_n}\pi(g)\otimes\pi(g^{-1})dg=F,
\]
in the weak operator topology, where $\Xi$ is the Harish-Chandra function  of $G$ and  $B_n$ is the ball of radius $n$ around the identity defined by a natural length function on $G$.  
\end{abstract}
\date{August  21, 2023}
\keywords{Furstenberg-Poisson boundary, Koopman representation, locally compact group, quasi-regular representation, Schur's orthogonality relations, semisimple group, unitary representation.}
\subjclass[2020]{Primary: 43A65 ; Secondary: 22E46}

\author{A. Bendikov}
\address{Institute of Mathematics, Wroclaw University}
\email{bendikov@math.uni.wroc.pl}
\author{A. Boyer}
\address{Université Paris Cité}
\email{adrien.boyer@imj-prg.fr }
\author{Ch. Pittet}
\address{I2M CNRS UMR7373 Aix-Marseille Univ. and Univ. of Geneva}
\email{pittet@math.cnrs.fr}
\thanks{This research was supported through the program ``Research in Pairs'' by the Mathematisches Forschungsinstitut Oberwolfach in January 2018. Ch. Pittet and A. Boyer acknowledge support of the FNS grant 200020-200400}

\maketitle

\tableofcontents
\section{Introduction}
Let $V$ and $W$ be vector spaces over a field $k$. According to the universal properties of the tensor product, the bilinear map 
$V\times W\to W\otimes V$, sending $(v,w)$ to $w\otimes v$, induces a unique vector space isomorphism 
\[
	V\otimes W\xrightarrow{\sigma_{V,W}}W\otimes V
\] 
such that $\sigma_{V,W}(v\otimes w)=w\otimes v$, whose inverse is $\sigma_{W,V}$.
If $G$ denotes a group, the above map $\sigma_{V,W}$ is an example of a \emph{symmetric braiding}, making the category of $k[G]$-modules, a \emph{symmetric tensorial category} \cite{ChaPre}. 
In this paper, we are interested in the case when $V=W=\mathcal H$ is a Hilbert space over the field of complex numbers, $G$ is a locally compact Hausdorff group, $\pi:G\to U(\mathcal H)$ is a unitary representation (we always assume that $\pi$ is strongly continuous: for any $v\in\mathcal H$, the map $g\mapsto\pi(g)v$ is continuous), and the symmetric braiding $\sigma_{\mathcal H,\mathcal H}=F$, called the \emph{flip} operator \cite[5.4]{Sav}, is the unique unitary involution, of the Hilbert tensor product of $\mathcal H$ with itself, characterized by the conditions
\[
	\mathcal H\otimes\mathcal H\xrightarrow{F}\mathcal H\otimes\mathcal H
\]
\[
	v\otimes w\mapsto w\otimes v,
\]
for all $v,w\in\mathcal H$. The flip  operator  has close links with the irreducibility of representations and their Schur's orthogonality relations as we now explain/recall\footnote{The equivalence: $\pi$ irreducible $\Leftrightarrow F\in W^*((\pi\otimes\pi)(G\times G))$, is probably folklore knowledge; we are grateful to U. Bader who mentioned it to us.}.
We denote $B(\mathcal H)$ the $C^*$-algebra of bounded operators on $\mathcal H$. If  $X\subset B(\mathcal H)$ is a subset, we denote its \emph{commutant} 
$$X'=\{y\in B(\mathcal H):\forall x\in X, yx=xy \}.$$
Assume $\pi$ is irreducible. It is well known that the irreducibility of $\pi$ is equivalent to the irreducibility of 
	            $$\pi\otimes\pi:G\times G\to U(\mathcal H\otimes\mathcal H).$$ 
Schur's lemma then implies that $(\pi\otimes\pi)(G\times G)^{'}$ is reduced to the multiples of the identity operator.
Hence $(\pi\otimes\pi)(G\times G)^{''}=B(\mathcal H\otimes\mathcal H)$ so it is trivial that $F\in(\pi\otimes\pi)(G\times G)^{''}$. 
Conversely, assume $F\in(\pi\otimes\pi)(G\times G)^{''}$.  The von Neumann double commutant theorem then implies that the braiding operator belongs to the weak closure of $(\pi\otimes\pi)(G\times G)$:
				$$F\in W^*((\pi\otimes\pi)(G\times G))=(\pi\otimes\pi)(G\times G)^{''}.$$
Consider  a closed proper $\pi(G)$-invariant subspace $V<\mathcal H$ and denote $V^{\perp}$ its orthogonal complement. As $\pi$ is unitary, the subspace $V^{\perp}$ is also $\pi(G)$-invariant, hence the closed subspace $V\otimes V^{\perp}$ of $\mathcal H\otimes\mathcal H$ is $(\pi\otimes\pi)(G\times G)$ invariant, hence $W^*((\pi\otimes\pi)(G\times G))$-invariant. In particular it is preserved by $F$. Hence
\[
V^{\perp}\otimes V=F(V\otimes V^{\perp})<	V\otimes V^{\perp},
\]
so that
\[
	V^{\perp}\otimes V< (V^{\perp}\otimes V)\cap(V\otimes V^{\perp})=\{0\}.
\]
This shows that either $V=\{0\}$ or $V^{\perp}=\{0\}$. Hence $\pi$ is irreducible.
To sum-up so far: a unitary representation $\pi:G\to U(\mathcal H)$ is irreducible if and only if 
the braiding operator $F$ is a limit, in the weak operator topology, of linear combinations of elements of $(\pi\otimes\pi)(G\times G)\subset B(\mathcal H\otimes\mathcal H)$. This raises the following general question: 
\emph{given an irreducible unitary representation
$\pi:G\to U(\mathcal H)$, what kind of linear combinations of elements of $(\pi\otimes\pi)(G\times G)$ are close/converge to $F$?}
An educated guess is to consider averages of elements of the type 
$\pi(g)\otimes{\pi}(g^{-1})$,
defined with the help of a sequence of finite symmetric Borel measures $\mu_n$ on $G$  (symmetric means $\mu_n(B^{-1})=\mu_n(B)$ for any Borel subset $B\subset G$). As we will explain, the body of known results leading to this guess can be organized as a succession of partial answers to the following general question/problem. 
\begin{ques}\label{general question easier}
Given an irreducible unitary representation $\pi:G\to U(\mathcal H)$, find families of symmetric measures
$(\mu_n)_{n>0}$ on $G$, and a dense subspace  $D<\mathcal H\otimes\mathcal H$, such that  for all $\alpha,\beta\in D$,
	\[
		\lim_{n\to\infty}\langle\int_{G}\pi(g)\otimes\pi(g^{-1})d\mu_n(g)\alpha,\beta\rangle=
		\langle F\alpha ,\beta\rangle.	
	\]	
\end{ques}
In order to link this question with previous works, about rapid decay (i.e. property RD), asymptotic Schur's orthogonality relations, and ergodic theorems on Furstenberg Poisson boundaries, we start with the following two general formulae, which are obvious to check. For all $v,w,v',w'\in\mathcal H$,
\begin{equation}\label{equation tau}
	\langle F(w'\otimes v),v'\otimes w\rangle=\langle v,v'\rangle\overline{\langle w,w'\rangle},
\end{equation} 
and if $\mu$ is symmetric, the change of variable $g\mapsto g^{-1}$ shows that, for any unitary representation (no irreducibility hypothesis here)
\begin{equation}\label{equation Schur}
\begin{split}
&\langle{\int_{G}\pi(g)\otimes{\pi}(g^{-1})d\mu(g)(w'\otimes v),v'\otimes w\rangle}=\\
&\int_{G}\langle\pi(g)v,w\rangle\overline{\langle\pi(g)v',w'\rangle}d\mu(g).
\end{split}	
\end{equation} 
The next proposition gathers various strengthenings of  Question \ref{general question easier} above.

\begin{prop}\label{prop:1,2,3} Let $\pi:G\to U(\mathcal H)$ be a unitary representation of a locally compact group and $\{\mu_n\}_{n>0}$ be a family of bounded symmetric Borel  measures on $G$. Among the five properties below, the following implications hold true: (1) $\Leftrightarrow$ (2) $\Rightarrow$ (3) $\Rightarrow$ (4) $\Rightarrow$ (5).
\begin{enumerate}
	\item\label{1} There is a uniform bound on the operator norms
	\[
	\sup_{n>0}\|\int_{G}\pi(g)\otimes{\pi}(g^{-1})d\mu_n(g)\|_{op}<\infty,	
	\]
	and a dense subset $V\subset\mathcal H$, such that for all $v,w,v',w'\in V$,
	\[
	\lim_{n\to\infty}\int_{G}\langle\pi(g)v,w\rangle\overline{\langle\pi(g)v',w'\rangle}d\mu_n(g)=
	\langle v,v'\rangle\overline{\langle w,w'\rangle}.
	\]
	\item\label{2} In the weak operator topology,
	\[
	\lim_{n\to\infty}\int_{G}\pi(g)\otimes{\pi}(g^{-1})d\mu_n(g)=F	
	\]
	\item\label{3} For all $v,w,v',w'\in \mathcal H$,
		\[
		\lim_{n\to\infty}\int_{G}\langle\pi(g)v,w\rangle\overline{\langle\pi(g)v',w'\rangle}d\mu_n(g)=
		\langle v,v'\rangle\overline{\langle w,w'\rangle}.
         \]
	\item\label{4} There is a dense subspace $V<\mathcal H$ such that
	for all $v,w,v',w'\in V$,
		\[
		\lim_{n\to\infty}\int_{G}\langle\pi(g)v,w\rangle\overline{\langle\pi(g)v',w'\rangle}d\mu_n(g)=
		\langle v,v'\rangle\overline{\langle w,w'\rangle}.
         \]
	\item\label{5} There is a dense subspace  $D<\mathcal H\otimes\mathcal H$, such that  for all $\alpha,\beta\in D$,
	\[
		\lim_{n\to\infty}\langle\int_{G}\pi(g)\otimes\pi(g^{-1})d\mu_n(g)\alpha,\beta\rangle=
		\langle F\alpha ,\beta\rangle.	
	\]	
\end{enumerate}
\end{prop}

\begin{rema} The implications $\mbox{(\ref{1})}\Rightarrow\mbox{(\ref{2})}\Rightarrow\mbox{(\ref{3})}$ are easily checked with the help of Formulae \emph{(\ref{equation tau})} and \emph{(\ref{equation Schur})}. The implication $\mbox{(\ref{2})}\Rightarrow\mbox{(\ref{1})}$ follows from  the Banach-Steinhaus theorem. 
We don't know if $\mbox{(\ref{3})}\Rightarrow\mbox{(\ref{2})}$. Notice that Property $\mbox{(\ref{3})}$ prevents the existence of a non-zero $\pi(G)$-invariant subspace $V<\mathcal H$ with $V^{\perp}$ non-zero. Hence any property among $\mbox{(\ref{1})},\mbox{(\ref{2})},\mbox{(\ref{3})}$, implies the irreducibility of $\pi$. If $V<\mathcal H$ is a subspace, we denote $V\otimes_{\emph{alg}}V$ the linear span in $\mathcal H\otimes\mathcal H$ of elements of the type $v\otimes w$ with $v,w\in V$. If $V$ is a given dense subspace of $\mathcal H$, then, obviously, Property $\mbox{(\ref{4})}$ with $V$, is equivalent to Property $\mbox{(\ref{5})}$ with $D=V\otimes_{\emph{alg}}V<\mathcal H\otimes\mathcal H$. 
\end{rema}

Any tempered  irreducible unitary representation $(\pi,\mathcal H)$ of a semisimple group $G$ with finite center over a local field, verifies Property (\ref{4}) of Proposition \ref{prop:1,2,3} above, with  $V<\mathcal H$ the subspace of $K$-finite vectors of a maximal compact subgroup $K$ of $G$, and
\[
	d\mu_n(g)=f\cdot\frac{\mathbf{1}_{B_n}(g)}{n^d}dg,
\]
where $d\in\mathbb N,\, f>0$ are constants depending on $\pi$ and $G$, and $B_n$ is the ball of radius $n$ around the identity element of $G$, with respect to a length function $L(g)=d(gx_0,x_0)$, defined with the help of a natural left $G$-invariant distance $d$, and a base point $x_0$, on the symmetric space, or - in the case the field is non-Archimedean -  on the building, associated to $G$; see Kazhdan and Yom Din \cite[Theorem 1.7.(1)]{KYD} in the case the local field is non-Archimedean, see Aubert and La Rosa \cite{AubLaR} in the case the local field is Archimedean.  Kazhdan and Yom Din conjecture \cite[Conjecture 1.2]{KYD} that the  convergence is not limited to the dense subspace of $K$-finite vectors, but holds in fact on the whole space $\mathcal H$; in other words, they conjecture that Property (\ref{3}) of Proposition \ref{prop:1,2,3} is true (for the measures $\mu_n$ described above).  They prove the conjecture in the case $G=SL(2,\mathbb R)$ or $G=PSL(2,\Omega)$, with $\Omega$ a non-Archimedean local field of characteristic zero and residual characteristic different from $2$,  \cite[Theorem 1.12]{KYD}. They also prove the conjecture in the case $\pi$ is the quasi-regular representation $L^2(G/P)$ (where $P$ a minimal parabolic subgroup), \cite[Theorem 1.11]{KYD}. Asymptotic Schur's orthogonality for semisimple Lie groups has also been considered by Midorikawa \cite{Mid1}, \cite{Mid2}, and Ankabout \cite[Formula 0.2]{Ank}. Some of their statements are equivalent to Property (\ref{4}) of Proposition \ref{prop:1,2,3}, with dense subspace $V$ and length function $L$, as in the work of Kazhdan and Yom Din just mentioned, but with measures
\[
	d\mu_n(g)=f\cdot\frac{e^{-L(g)/n}}{n^d}dg,
\]
where $d$ is an integer and $f>0$.

According to Bader and Muchnik \cite{BadMuc} the natural \emph{boundary representation} $\pi$ of the fundamental group $\Gamma$ of a closed Riemannian manifold $X$, with strictly negative sectional curvature, is irreducible. (Recall that $\Gamma$ acts on the Gromov boundary $\partial \tilde{X}$ of the universal cover $\tilde{X}$ of $X$ and preserves the class of a normalized Patterson-Sullivan measure $\nu$ on $\partial \tilde{X}$. The representation $\pi$ is defined as the Koopman representation on $L^2(\partial \tilde{X},\nu)$ associated to this action.)
Garncarek has generalized this result to all non-elementary Gromov hyperbolic groups \cite[Theorem 2.1]{BoyGar}. In \cite{BoyGar}, Boyer and Garncarek prove that boundary representations (associated to Patterson-Sullivan measures and also to Gibbs streams) of non-elementary Gromov hyperbolic groups satisfy asymptotic Schur's orthogonality relations. For example, in the special case $\Gamma$ is a non elementary discrete convex cocompact group of isometries of a proper $CAT(-1)$ space  (with non-arithmetic spectrum), they show that Property (\ref{3}) of Proposition \ref{prop:1,2,3} is true for the boundary representation $\pi:\Gamma\to U(L^2(\partial \Gamma,\nu))$, and with the sequence of measures
\[
	d\mu_n(\gamma)=\frac{\mathbf{1}_{A_n}(\gamma)}{|A_n|\langle\pi(\gamma)\mathbf{1}_{\partial \Gamma},\mathbf{1}_{\partial \Gamma}\rangle^2}d\gamma,
\]
where the Haar measure $d\gamma$ is the counting measure,  $A_n=B_{n+h}\setminus B_{n-h}$ is a sequence of thick enough (i.e. $h>>1$) annuli around the identity in $\Gamma$, defined, for $n$ large enough, with the help of a left $\Gamma$-invariant metric on $\Gamma$, the cardinality of $A_n$ is denoted $|A_n|$, 
and  the diagonal coefficient $\langle\pi(\gamma)\mathbf{1}_{\partial \Gamma},\mathbf{1}_{\partial \Gamma}\rangle$ of $\pi$, defined by the constant function equal to $1$ on $\partial \Gamma$, is the analogous of the classical Harish-Chandra function of non-compact semisimple groups with finite center. An important ingredient in the proof of the main theorem of \cite{BoyGar} is that the restrictions of  the double boundary map  $\Gamma\to\partial\Gamma\times\partial\Gamma$, $g\mapsto(\hat{g},\check g)$, to the nested balls around the identity, form an equidistributed sequence relative to the product measure on the product of the boundary with itself \cite[3.1]{BoyGar}. This dynamical property of groups actions on boundaries has  been encountered and formalized by many authors, e.g. Roblin for $CAT(-1)$ spaces and groups (excluding free groups) \cite[Théorème 4.1.1]{Rob}, Boyer and Pinochet for free groups \cite[Theorem 1.1]{BoyPin}, Paulin and Pollicott and Schapira  \cite[Theorem 1.4]{Pau}, Bader and Muchnik \cite[Corollary 5.4]{BadMuc}, for manifolds of negative curvature,  Boyer and Link and Pittet for lattices in semisimple groups \cite[Lemma 3.4]{BoyLinPit}. 
All these results reveal a dynamical/geometric aspect  of Schur's orthogonality relations and motivate further the choice of  operators of the form $\pi(g)\otimes\pi(g^{-1})$ in Question \ref{general question easier}. Asymptotic Schur's orthogonality for boundary representations of finitely generated non-abelian free groups has also been considered by Kuhn and Steger \cite[Corollary 3.4]{KuhSte}, Pensavalle and Steger \cite[Theorem 1.2, Theorem 2.9]{PenSte}. Some of their statements are equivalent to Property (\ref{4}) of Proposition \ref{prop:1,2,3}, with measures
\[
	d\mu_n(\gamma)=f\cdot\frac{e^{-L(\gamma)/n}}{n}d\gamma,
\]
where $f>0$, and $L(\gamma)$ is the word length of $\gamma$ with respect to a finite symmetric generating set. 

We suspect that Property (\ref{4}) of Proposition \ref{prop:1,2,3} is true for the quasi-regular representation $\pi$ of a lattice $\Gamma$ (uniform or not) in a non-compact semisimple Lie group $G$ with finite center. We hope to elaborate on this matter in a forthcoming paper.  

A unitary representation $\pi:G\to U(\mathcal H)$ of a locally compact Hausdorff group has \emph{property RD} with respect to a length function $L$ if there is a polynomial $P(X)\in\mathbb R[X]$ such that for any $v,w\in\mathcal H$ such that $\|v\|\leq 1, \|w\|\leq 1$, any $n\in\mathbb N$,
\[
	\int_{B_n}|\langle\pi(g)v,w\rangle|^2 dg\leq P(n),
\]
where $B_n$ is the ball of radius $n$ in $G$, with center the identity, defined by the length function $L$. See \cite[Proposition 4.1]{Cha}.
Notice that if $\pi$ satisfies Property (\ref{3}) of Proposition \ref{prop:1,2,3} with
\[
d\mu_n(g)=f\cdot\frac{\mathbf{1}_{B_n}(g)}{n^d}dg,	
\] 
then it satisfies property RD: choosing $\alpha=w\otimes v$ and $\beta=v\otimes w$, with $\|v\|=\|w\|=1$, we see that
\[
	\lim_{n\to\infty}\frac{1}{n^d}\int_{B_n}|\langle\pi(g)v,w\rangle|^2 dg=\frac{1}{f},
\]
and we conclude with the help of the Banach-Steinhaus theorem.
Applying this implication together with Theorem \ref{main result}  below, one recovers the well-known fact that semisimple groups have property RD.
See \cite{ChaPitSal} and \cite{BoyCRAS} for different proofs of property RD for semisimple groups. If the quasi-regular representation of a uniform lattice $\Gamma$ in a semisimple Lie group satisfies Property (\ref{3}) of Proposition \ref{prop:1,2,3} with  
\[
d\mu_n(\gamma)=f\cdot\frac{\mathbf{1}_{B_n}(\gamma)}{n^d}d\gamma,	
\] 
where the ball $B_n$ is defined with a length function, then Valette's conjecture is true, i.e. uniform lattices in semisimple Lie groups have property RD \cite[Conjecture 7]{Val}, \cite[Conjecture 1.1]{Cha}. We don't know if Property (\ref{3}) is true for uniform lattices in higher rank, even when the Valette conjecture is verified (see \cite[page 55]{Cha}). Nota bene: the quasi-regular representation of a non-uniform lattice in a semisimple Lie group of rank at least $2$ does not satisfy property RD (the lattice contains exponentially distorted elements) \cite[Proposition 8.7]{Val}. 

The main result of this paper is that Property (\ref{2}) of Proposition \ref{prop:1,2,3} is true for semisimple groups. More precisely, consider the following setting. Let $G$ be a semisimple group over a local field. We assume that $G$ is non compact with finite center. Let  $P<G$ be a minimal parabolic subgroup and $K<G$ be a maximal compact subgroup such that $G=KP$. Let $\nu$ be the unique regular Borel probability measure on
$G/P=K/(K\cap P)$ which is $K$-invariant. Let $\pi:G\to U(L^2(G/P,\nu))$ be the quasi-regular representation.
Let $(X,d)$ be either the Riemannian symmetric space with its usual $G$-invariant metric $d$ defined by the Killing form, or - in the case the local field is non-Archimedean - the Bruhat-Tits building associated to $G$, with a $G$-invariant metric $d$, making the apartments totally geodesic and isometric to an Euclidean space.  Let $x_0\in X$ and  denote $L(g)=d(x_0,gx_0)$ the corresponding length function on $G$ and $B_n=\{g\in G: L(g)\leq n\}$. 

\begin{theo}\label{main result}(Convergence to the flip for the quasi-regular representation of a semisimple group.)
With the notation as above, there exists $d=d(G)\in\mathbb N$ and $f=f(G)>0$, such that, in the weak operator topology,
\[
	\lim_{n\to\infty}\int_G\pi(g)\otimes\pi(g^{-1})d\mu_n(g)=F,
\]
for the measures
\[
	d\mu_n(g)=f\cdot\frac{\mathbf{1}_{B_n}(g)}{n^d}dg.
\]	
In other words, for any  $\alpha,\beta\in L^2(G/P,\nu)\otimes L^2(G/P,\nu)$,
	\[
		\lim_{n\to\infty}\langle\int_{G}\pi(g)\otimes\pi(g^{-1})d\mu_n(g)\alpha,\beta\rangle=
		\langle F\alpha ,\beta\rangle.	
	\]	
\end{theo}
We will prove that Property (\ref{1}) of Proposition \ref{prop:1,2,3} is true in the setting of the theorem (recall that Property (\ref{1}) is equivalent to (\ref{2})).
The method of the proof works in a  setting more general than the one of the theorem. For example, it also applies easily to the quasi-regular representation $\pi$ of the automorphism group of a regular tree $T$ of degree $q+1$: it implies that 
\[
	\lim_{n\to\infty}\frac{1}{n^3}\int_{B_n}\pi(g)\otimes\pi(g^{-1})dg=\frac{1}{3}\frac{(q-1)^2}{q(q+1)}F,
\]
where  $F\in U(L^2(\partial T,\nu)\otimes L^2(\partial T,\nu))$ is the flip and $\nu$ is the probability Hausdorff measure on the Gromov boundary $\partial T$, and $B_n$ is the ball of radius $n$ around the identity, defined by the length function associated to the combinatorial length distance on $T$  (in which each edge has length $1$). In Theorem \ref{theo:uniform} below, the uniform bound in Property (\ref{1}) of Proposition \ref{prop:1,2,3}, on the norms of the operators
$M_n=\int_{G}\pi(g)\otimes\pi(g^{-1})d\mu_n(g)$,
is obtained with the help of Riesz-Thorin's interpolation: we prove that 
\[
	\|M_n\|_{2\to 2}\leq\|M_n\|_{\infty\to\infty},
\]
and, in the case $G$ is semisimple or  $G=\mbox{Aut}(T)$, that $\|M_n\|_{\infty\to\infty}=1$ for all $n$. The existence of a compact subgroup $K<G$, acting transitively on the boundary is essential: if $G=\mathbb F_r$ is a non-abelian free group of finite rank $r$, acting on its Gromov boundary, and if $\pi$ is the associated boundary representation, then the corresponding operators satisfy
$$\lim_{n\to\infty}\|M_n\|_{\infty\to\infty}=\infty,$$
hence the above argument to bound uniformly $\|M_n\|_{2\to 2}$ breaks down and we do not know if Property (\ref{2}) of Proposition \ref{prop:1,2,3} is true in this case.



\section{Outline of the paper}
In Section \ref{section: compact} we briefly remind the reader the classical Schur's orthogonality relations for the coefficients of representations of compact groups: this is an absolute classic, but proving it using the braiding operator  is fun and prepares the ground for the non-compact cases. In Section \ref{section: convergence to the braiding operator} we establish an ``asymptotic Schur's lemma'' for symmetric operators, which is a variation - suitable for the braiding operator - on the concept of \emph{Folner $c$-temperness}, from Kazhdan and Yom-Din \cite{KYD}. In Section \ref{Koopman representations and cocycle averages} we study cocycles averages and apply Riesz-Thorin's interpolation to obtain a uniform bound on the norms of averages of the type $$\int_G\pi(g)\otimes\pi(g^{-1})d\mu_n,$$ when $\pi$ is a Koopmann representation. In Section \ref{putting things together} we rely on Section \ref{section: convergence to the braiding operator} and Section \ref{Koopman representations and cocycle averages} to obtain a criterion for the convergence to the braiding operator. In the last section, we rely on estimates of the Harish-Chandra function of a semisimple group, due to Kazhdan and Yom-Din \cite{KYD}, and we apply the criterion from Section \ref{putting things together} to semisimple groups over local fields, proving Theorem \ref{main result} stated in the Introduction. We also consider automorphisms of regular trees. In the appendix, we present elementary facts about closable  operators, needed in the proofs from Section \ref{section: convergence to the braiding operator}.

\emph{Acknowledgements.}
We warmly thank Uri Bader and Jean-Fran\c cois Quint for helpful discussions. We are indebted to Anne-Marie Aubert and Alfio Fabio La Rosa who pointed out an erroneous citation in a previous version of the paper and mentioned to us their work on the subject.  This work was initiated during a RIP stay of the three authors at the  Mathematisches Forschungsinstitut Oberwolfach (MFO). We are very grateful to all the people involved in the running of the MFO for providing us with wonderful working conditions. Adrien Boyer and Christophe Pittet are grateful to Enrico Leuzinger, Roman Sauer, Tobias Hartnick, Claudio Llosa Isenrich, for having invited them together at the KIT.

\section{The compact case}\label{section: compact}
In the case $G$ is compact, Question \ref{general question easier} has a neat answer which brings a short proof of the Schur's orthogonality relations for the coefficients of an irreducible representation. As a warm-up, we review the compact case: 

\begin{prop}(Braiding and Schur's orthogonality for the coefficients of an irreducible representation of a compact group.) Let  $\pi:G\to U(\mathcal H)$ be an irreducible unitary representation of a compact group.
\begin{enumerate}
	\item The dimension $d$ of $\mathcal H$ is finite.
	\item The probability Haar measure $dg$ on $G$ satisfies
	\[
	\int_{G}\pi(g)\otimes{\pi}(g^{-1})dg=\frac{1}{d}F.	
	\]
	\item For any $v,w,v',w'\in\mathcal H$,
	\[
	\int_{G}\langle\pi(g)v,w\rangle\overline{\langle\pi(g)v',w'\rangle}dg=\frac{1}{d}\langle v,v'\rangle\overline{\langle w,w'\rangle}.	
	\]
\end{enumerate}
\end{prop}
\begin{proof} We denote $M=\int_{G}\pi(g)\otimes{\pi}(g^{-1})dg$ and define $A=FM$.
For any $g,h\in G$, it is obvious that $(\pi(g)\otimes{\pi}(h))F=F(\pi(h)\otimes{\pi}(g))$, and simple changes of variables show that $(\pi(g)\otimes{\pi}(h))M=M(\pi(h)\otimes{\pi}(g))$, hence $(\pi(g)\otimes{\pi}(h))A=A(\pi(g)\otimes{\pi}(h))$. Schur's lemma, applied to the irreducible representation 
$\pi\otimes\pi:G\times G\to U(\mathcal H\otimes\mathcal H)$, implies that there exists $c\in\mathbb C$ such that $A=cI$. As $F$ is an involution we obtain $M=cF$. In particular, for any $v,w,v',w'\in\mathcal H$, 
\[
	\langle M(w'\otimes v),v'\otimes w\rangle=c\langle F(w'\otimes v),v'\otimes w\rangle.
\]
In other words,
	\[
	\int_{G}\langle\pi(g)v,w\rangle\overline{\langle\pi(g)v',w'\rangle}dg=c\langle v,v'\rangle\overline{\langle w,w'\rangle}.	
	\]
Choosing $v=v'=w=w'$ of norm $1$ (we assume $\mathcal H\neq\{0\}$), we see that 
\[
	c=\int_{G}|\langle\pi(g)v,v\rangle|^2dg>0.
\]
Assume $n\in\mathbb N$ is such that there exists an orthonormal family $e_1,\dots,e_n,$ in $\mathcal H$. We choose $v\in\mathcal H\setminus\{0\}$. Integrating the inequality
\[
\sum_{i=1}^n|\langle\pi(g)v,e_i\rangle|^2\leq\|\pi(g)v\|^2=\|v\|^2,	
\]
we obtain
\begin{align*}
	nc\|v\|^2&=\sum_{i=1}^nc\langle v,v\rangle\overline{\langle e_i,e_i\rangle}
	         =\sum_{i=1}^n\int_{G}|\langle\pi(g)v,e_i\rangle|^2dg\leq \|v\|^2.
\end{align*}
This shows that $n\leq 1/c$. Hence $\mathcal H$ is of finite dimension $d$.	Choosing $n$ maximal, i.e. $n=d$, we see that the above inequalities are equalities, hence $dc=1$.
\end{proof}

\section{Convergence to the braiding operator}\label{section: convergence to the braiding operator}
Proposition \ref{Identification of the limit} and Example \ref{Folner c-temperness} below are variations, suitable for the braiding operator, on the concept of \emph{Folner $c$-temperness}, from Kazhdan and Yom-Din \cite[Proposition 2.3]{KYD}. The following obvious lemma will be applied in the course of the proof of Proposition \ref{Identification of the limit}.

\begin{lemma}\label{lemma obvious}
Let $\mathcal H$ be a complex Hilbert space. Let $U,V\in B(\mathcal H)$, and $\{M_n\}_{n>0}$ a sequence in $B(\mathcal H)$. Let $D<\mathcal H$ be a dense subspace and  $M:D\to\mathcal H$ a linear operator. We assume that $U^*$ and $V$ preserve $D$.
If $M_n$ weakly converges to $M$ on $D$, that is, for any $v,w\in D$,
\[
	\lim_{n\to\infty}\langle M_nv,w\rangle=\langle Mv,w\rangle,
\]
and if $UM_n-M_nV$ weakly converges to zero on $D$, then $UM-MV$ vanishes on $D$.
\end{lemma}

\begin{proof} Let $v,w\in D$. For each $n>0$,
	\begin{align*}
	\langle(UM-MV)v,w\rangle&=\langle(M-M_n)v,U^*w\rangle+\langle(UM_n-M_nV)v,w\rangle\\
	                        &+\langle(M_n-M)Vv,w\rangle,	
	\end{align*}
and each term on the right-hand side goes to zero when $n$ goes to infinity. This shows that $(UM-MV)v\in D^{\perp}=\overline{D}^{\perp}=\mathcal H^{\perp}=\{0\}$.
\end{proof}

\begin{prop}\label{Identification of the limit}(Identification of the limit.) Let $\pi:G\to U(\mathcal H)$ be an irreducible unitary representation of a locally compact group.
Let $S\subset G$ be a generating set of $G$ (i.e. the smallest subgroup containing $S$ is $G$). Let $\{\mu_n\}_{n>0}$ be a family of finite symmetric regular Borel measures on $G$. Let $D<\mathcal H\otimes\mathcal H$ be a dense subspace preserved by $\pi\otimes\pi$. Assume the family of operators 
\[
M_n=\int_G\pi(g)\otimes\pi(g^{-1})d\mu_n(g)\in B(\mathcal H\otimes\mathcal H)
\]
weakly converges  on $D$ to an (a priori not necessary bounded) operator $M:D\to\mathcal H\otimes\mathcal H$. That is: for any $\alpha,\beta\in D$,
\[
	\lim_{n\to\infty}\langle M_n\alpha,\beta\rangle=\langle M\alpha,\beta\rangle.
\]
Assume also that for any $s\in S$, 
\[
	\int_G\pi(g)\otimes\pi(g^{-1})d\mu_n(sg)-\int_G\pi(g)\otimes\pi(g^{-1})d\mu_n(g),
\]
as well as
\[
	\int_G\pi(g)\otimes\pi(g^{-1})d\mu_n(gs)-\int_G\pi(g)\otimes\pi(g^{-1})d\mu_n(g),
\]
weakly converge  to zero on $D$. Then the operator $M$ is bounded and its unique continuous extension to $\mathcal H\otimes\mathcal H$ equals a multiple $c\cdot F$ of the braiding operator. Hence for any $\alpha,\beta\in\mathcal H\otimes\mathcal H$,
\[
\lim_{n\to\infty}\langle \int_G\pi(g)\otimes\pi(g^{-1})d\mu_n(g)\alpha,\beta\rangle=c\langle F\alpha,\beta\rangle.	
\]
In particular, if $\alpha=w'\otimes v$ and $\beta=v'\otimes w$ belong to $D$, then
\[
\lim_{n\to\infty}\int_G\langle\pi(g)v,w\rangle\overline{\langle\pi(g)v',w'\rangle}d\mu_n(g)=c\langle v,v'\rangle\overline{\langle w,w'\rangle}.
\]
\end{prop}

\begin{proof} Let $s\in S$. We claim that 
	\[
	\lim_{n\to\infty}[(\pi(s^{-1})\otimes I)M_n-M_n(I\otimes\pi(s^{-1}))]=0	
	\]
on $D$, in the weak operator topology. In order to prove this claim, notice that the change of variable $h=s^{-1}g$ leads to the equalities
\begin{align*}
&(\pi(s^{-1})\otimes I)M_n-M_n(I\otimes\pi(s^{-1}))=\\
&\int_G\pi(s^{-1}g)\otimes \pi(g^{-1})d\mu_n(g)-M_n(I\otimes\pi(s^{-1}))=\\
&\int_G\pi(h)\otimes \pi(h^{-1})d\mu_n(sh)(I\otimes\pi(s^{-1}))-M_n(I\otimes\pi(s^{-1}))=\\
&\left[\int_G\pi(g)\otimes \pi(g^{-1})d\mu_n(sg)-M_n\right](I\otimes\pi(s^{-1})).						
\end{align*}
This proves the claim. Similarly,
	\[
	\lim_{n\to\infty}[(I\otimes\pi(s^{-1}))M_n-M_n(\pi(s^{-1})\otimes I)]=0.	
	\]
Applying Lemma \ref{lemma obvious} we deduce that on $D$, $$(\pi(s^{-1})\otimes I)M=M(I\otimes\pi(s^{-1})),$$ as well as $$(I\otimes\pi(s^{-1}))M=M(\pi(s^{-1})\otimes I).$$ As $S$ generates $G$, we deduce that for any $g,h\in G$,
\[
(\pi(g)\otimes\pi(h))M=M(\pi(h)\otimes\pi(g)),	
\]
on $D$. It is obvious that for any  $g,h\in G$,
\[
(\pi(g)\otimes\pi(h))F=F(\pi(h)\otimes\pi(g)).		
\]
Eventually we obtain, for all $g,h\in G$,
\begin{equation}\label{commutes}
(\pi(g)\otimes\pi(h))F M(\pi(g)\otimes\pi(h))^{-1}=F M,	
\end{equation}
on $D$.
As $\mu_n$ is symmetric, $M_n=M_n^*$. It follows that the limit $M$ is symmetric:
\[
	\forall \alpha,\beta\in D,\, \langle M\alpha,\beta\rangle=\langle \alpha,M\beta\rangle.
\]
Let $\overline{M}:D\left(\overline{M}\right)\to\mathcal H\otimes\mathcal H$ be the closure of $M$ (as $M$ is symmetric, the closure of its graph is again a graph - the one of $\overline{M}$). The operator  $F M$ is closable with closure $\overline{F M}=F\overline{M}$ (see Lemma \ref{closure 1} from the appendix), and its domain $D(\overline{F M})=D(\overline{M}$)  is preserved by $\pi\otimes\pi$ (see Lemma \ref{closure 2} from the appendix). For all $g,h\in G$,
\[
(\pi(g)\otimes\pi(h))\overline{F M}(\pi(g)\otimes\pi(h))^{-1}=\overline{F M},
\]
on $D(\overline{M})$ (as $\pi(g)\otimes\pi(h)$ is a unitary transformation of $\mathcal H\otimes\mathcal H$, this follows from Formula (\ref{commutes}) above and Lemma \ref{closure 2} from the appendix). As $\pi\otimes\pi$ is irreducible, Schur's Lemma for closed operators applies and shows that
$\overline{F M}=F\overline{M}$ is bounded and that its unique continuous extension to $\mathcal H\otimes\mathcal H$ equals $c\cdot I$ where $c$ is a constant. As $F$ is an involution, we eventually conclude that $M$ is bounded and that its unique continuous extension to $\mathcal H\otimes\mathcal H$ equals $c\cdot F$. 
We have just established that for any $\alpha,\beta\in D$,
\[
\lim_{n\to\infty}\langle \int_G\pi(g)\otimes\pi(g^{-1})d\mu_n(g)\alpha,\beta\rangle=c\langle F\alpha,\beta\rangle.	
\]
Applying Formulae (\ref{equation tau}) and (\ref{equation Schur}) finishes the proof. 
\end{proof}

\begin{exam}\label{Folner c-temperness}(Asymptotically invariant operators.) Let $\pi:G\to U(\mathcal H)$ be a unitary representation of a locally compact unimodular group. For each $n\in\mathbb N$, let $F_n\subset G$ be a relatively compact Borel set and $P(n)>0$ a positive number. Consider the sequence of measures
\[
	d\mu_n(g)=\frac{\mathbf{1}_{F_n}(g)}{P(n)}dg.
\]
Let $D<\mathcal H$ be a dense subspace. We denote $D\otimes_{\emph{alg}}D$ the linear span in $\mathcal H\otimes\mathcal H$ of elements of the type $v\otimes w$ with $v,w\in D$. Let $s,t\in G$. Suppose for all $v,w\in D$,
	\[
		\lim_{n\to\infty}\frac{1}{P(n)}\int_{(sF_nt)\setminus F_n}|\langle\pi(g)v,w\rangle|^2dg=0.
	\]	
Then
\[
\int_G\pi(g)\otimes\pi(g^{-1})d\mu_n(t^{-1}gs^{-1})-\int_G\pi(g)\otimes\pi(g^{-1})d\mu_n(g) 
\]
weakly converges to zero on $D\otimes_{\emph{alg}}D$.	
\end{exam}

\begin{proof} According to Formula (\ref{equation Schur}), it is enough to show that $\forall v,w,v',w'\in D$, the difference
	\[
		\int_G\langle\pi(g)v,w\rangle\overline{\langle\pi(g)v',w'\rangle}d\mu_n(t^{-1}gs^{-1})-\int_G\langle\pi(g)v,w\rangle\overline{\langle\pi(g)v',w'\rangle}d\mu_n(g)
	\]
goes to zero as $n$ goes to infinity.
As 	$\mathbf{1}_{F_n}(s^{-1}gt^{-1})=\mathbf{1}_{sF_nt}(g)$, the above difference equals 
\[
	\frac{1}{P(n)}\int_{(sF_nt)\setminus F_n}\langle\pi(g)v,w\rangle\overline{\langle\pi(g)v',w'\rangle}dg.
\]
Taking norms and applying the Cauchy-Schwarz inequality finishes the proof.
\end{proof}

\section{Koopman representations and cocycle averages}\label{Koopman representations and cocycle averages}
The goal of this section is to provide a framework for bounding operator norms of averages of the type
\[
	\int_G\pi(g)\otimes\pi(g^{-1})d\mu(g),
\]
when the unitary representation $\pi$ is a quasi-regular representation and the group $G$ has a compact subgroup $K$ acting transitively by  measure preserving transformations. Typical examples of such situations are: $G$ a semisimple group with finite center and $\pi$ the quasi-regular representation on $L^2(G/P)$ where $P$ is a minimal parabolic of $G$, or $G=\mbox{Aut}(T)$ where $T$ is a regular tree of bounded degree and   $\pi$ is the quasi-regular representation on $L^2(\partial T)$, where $\partial T$ is the Gromov boudary of $T$ with its Bourdon metric and probability Hausdorff measure.

\begin{prop}\label{prop: cocycle average} Let $G$ be a locally compact  group acting by measurable transformations on a probability space $(X,\nu)$. We assume the action preserves the class of $\nu$. Let $\mu$ be a positive Borel measure on $G$. Let
\[
	w:G\times X\to\mathbb C
\]
be a a measurable cocycle. We assume that $(g,x)\mapsto w(g^{-1},x)$ and $(g,x)\mapsto w(g^{-1},x)$ are $(\mu\otimes\nu)$-integrable, and that
$$(g,x,y)\to w(g^{-1},x)w(g,y)$$ is $(\mu\otimes\nu\otimes\nu)$-integrable. Then the following is true.
\begin{enumerate}
	\item The integral
	\[
		\Xi_w(g)=\int_Xw(g^{-1},x)d\nu(x),
	\]
	is well defined $\mu$-a.e. The integral
	\[
		\mathcal A(x,y)=\int_Gw(g^{-1},x)w(g,y)d\mu(g),
	\]
	is well defined $(\nu\otimes\nu)$-a.e.  and 
	\[
	\int_X\int_X\mathcal A(x,y)d\nu(x)d\nu(y)=\int_G\Xi_w(g)\Xi_w(g^{-1})d\mu(g).	
	\]
	\item Assume that $K<G$ is a subgroup such that the restriction of $w$ to $K$ is trivial. If $\mu$ is $K$-bi-invariant, then for any $k,l\in K$, 
$$\mathcal A(kx,ly)=\mathcal A(x,y)$$
holds $(\nu\otimes\nu)$-a.e.
	\item If moreover the action of $K$ on $X$ is transitive, then $\mathcal A$ is essentially constant, more precisely: 
	\[
	\mathcal A(x,y)=\int_G\Xi_w(g)\Xi_w(g^{-1})d\mu(g)
	\]
	holds $(\nu\otimes\nu)$-a.e.
\end{enumerate}
\end{prop}

\begin{proof} Point $(1)$ follows from Fubini's theorem (see for example \cite[Theorem C, Sec. 36]{Hal}). To prove $(2)$, notice first that for any $k\in K$, $g\in G$, the cocycles identities
	\[
		w(gk,x)=w(g,kx)w(k,x)=w(g,kx),
	\]
	\[
	    w(kg,x)=w(k,gx)w(g,x)=w(g,x),	
	\]
hold $\nu$-a.e. in $X$ (or everywhere if the cocycle is strict).
We choose $k\in K$ and perform the change of variable $s=gk$; using the right $K$-invariance of $\mu$, we obtain $(\nu\otimes\nu)$-a.e.:
\begin{align*}
	\int_Gw(g^{-1},x)w(g,ky)d\mu(g)&=\int_Gw(g^{-1},x)w(gk,y)d\mu(g)\\
						&=\int_Gw(ks^{-1},x)w(s,y)d\mu(s)\\
						&=\int_Gw(s^{-1},x)w(s,y)d\mu(s).
\end{align*}
This proves $K\times\{e\}$-invariance. A similar computation, using the left $K$-invariance of $\mu$, proves $\{e\}\times K$-invariance.
Point $(3)$ is easily deduced from points $(1)$ and $(2)$ because of the following obvious fact: the essentially constant function $\mathcal A(x,y)$  on the probability space $(K\times K,\nu\otimes\nu)$  is $(\nu\otimes\nu)$-a.e. equal to its integral.	
\end{proof}

The proof of the following theorem uses the above proposition. The theorem applies to representations from the principal series of semisimple groups and of automorphism groups of regular trees. In the case of quasi-regular representations (i.e. when the cocycle takes only real values) its conclusion is optimal. 

\begin{theo}\label{theo:uniform}(Uniform boundedness for symmetric means of operators associated to the principal series.) Suppose a locally compact   group $G$ acts by mesurable transformations on a probability space $(X,\nu)$. We assume that the action preserves the class of $\nu$.
For $g\in G$, we denote
$$x\mapsto c(g,x)=\sqrt{\frac{d(g^{-1}_*\nu)}{d\nu}(x)}$$ the cocycle defined by the square root of the Radon-Nikodym derivative of the action of $g^{-1}$.
Let $u:G\times X\to U(1,\mathbb C)$  be a measurable cocycle with values in the group of complex numbers of norm $1$. We assume the associated representation
$$\pi:G\to U(L^2(X,\nu))$$
$$(\pi(g)\varphi)(x)=\varphi(g^{-1}x)u(g^{-1},x) c(g^{-1},x)$$
is strongly continuous. Let $\mu$ be a finite symmetric Borel measure on $G$. We assume that $(g,x)\mapsto  c(g^{-1},x)$ and $(g,x)\mapsto  c(g,x)$ are $(\mu\otimes\nu)$-integrable, and that
$$(g,x,y)\to  c(g^{-1},x) c(g,x)$$ is $(\mu\otimes\nu\otimes\nu)$-integrable. Let
\[
	m:L^2(X,\nu)\otimes L^2(X,\nu)\to L^2(X\times X,\nu\otimes\nu)
\]
be the isomorphism of Hilbert spaces, uniquely determined by the condition $m(\varphi\otimes\psi)=\varphi\cdot\psi$, and
let
\[
	\sigma(g)=m(\pi(g)\otimes\pi(g^{-1}))m^{-1}.
\]
We assume that for $p=1,2,\infty$, 
\[
\int_G\sigma(g)d\mu(g)\in B(L^p(X\times X,\nu\otimes\nu)). 	
\]
Then the norm of the $\mu$-average of the operators $\pi(g)\otimes\pi(g^{-1})\in B(L^2(X,\nu)\otimes L^2(X,\nu))$ satisfies
\[
	\|\int_G\pi(g)\otimes\pi(g^{-1})d\mu(g)\|_{op}\leq\sup_{(x,y)}\int_Gc(g^{-1},x)c(g,y)d\mu(g).
\]
If moreover, there is a subgroup $K<G$ which acts transitively on $X$ with fixed vector $\mathbf{1}_X$ (i.e. for all $k\in K$, $\pi(k)\mathbf{1}_X=\mathbf{1}_X$), and if $\mu$ is $K$-bi-invariant,
then 
\[
	\|\int_G\pi(g)\otimes\pi(g^{-1})d\mu(g)\|_{op}\leq\int_G\Xi(g)^2d\mu(g),
\]
where 
\[
\Xi(g)=\int_X c(g^{-1},x)d\nu(x),
\]
and in the special case the cocycle $u$ is trivial (i.e. $u(g,x)=1$ holds $\mu\otimes\nu$-a.e.), then the above inequality is an equality.
\end{theo}

\begin{proof} The operator $\int_G\pi(g)\otimes\pi(g^{-1})d\mu(g)\in B(L^2(X,\nu)\otimes L^2(X,\nu))$ is self-adjoint because $\mu$ is symmetric.  Hence, its conjugate by the isometry $m$, 
is self-adjoint and belongs to $B(L^p(X\times X,\nu\otimes\nu))$ for $p=1,2,\infty$. Riesz-Thorin's interpolation theorem implies
\[
	\|\int_G\sigma(g)d\mu(g)\|_{2\to 2}\leq\|\int_G\sigma(g)d\mu(g)\|_{\infty\to\infty}.
\]
For any $\varphi\in L^{\infty}(X\times X,\nu\otimes\nu)$ such that $\|\varphi\|_{\infty}=1$, we have ($\sup$ means the essential supremum),
\begin{align*}
&\sup_{(x,y)}|\left(\int_G\sigma(g)\varphi d\mu(g)\right)(x,y)|=\\
&\sup_{(x,y)}|\int_G\varphi(g^{-1}x,gy)u(g^{-1},x)c(g^{-1},x)u(g,y)c(g,y)d\mu(g)|\leq\\	
&\sup_{(x,y)}|\int_G|\varphi(g^{-1}x,gy)||u(g^{-1},x)|c(g^{-1},x)|u(g,y)|c(g,y)d\mu(g)|\leq\\	
&\sup_{(x,y)}\int_Gc(g^{-1},x)c(g,y)d\mu(g).	
\end{align*}
Under the hypothesis of the existence of $K$, applying Proposition \ref{prop: cocycle average}, we deduce that the function
\[
	(x,y)\mapsto \int_Gc(g^{-1},x)c(g,y)d\mu(g)
\]
is essentially constant equal to $\int_G\Xi(g)^2d\mu_(g)$.
In the case $u$ is trivial, taking $\varphi=\mathbf{1}_{X\times X}$ we get $\nu\otimes\nu$-a.e.
\[
	\left(\int_G\sigma(g)\mathbf{1}_{X\times X}d\mu(g)\right)(x,y)=\int_G\Xi(g)^2d\mu(g).
\]	
\end{proof}


	
\section{Putting things together}\label{putting things together}
Recall that a length function $L:G\to[0,\infty[$ on a locally compact group is a Borel map such that $L(e)=0$, and for all $g,h\in G$, $L(g^{-1})=L(g)$, $L(gh)\leq L(g)+L(h)$.
We assume the inverse image by $L$ of any bounded subset is relatively compact. We denote $$B_n=L^{-1}\left([0,n]\right)$$ the associated ball of radius $n>0$ around the identity.

The following corollary is a consequence of Theorem \ref{theo:uniform} combined with Proposition \ref{Identification of the limit}.

\begin{coro}\label{coro}(Convergence to the braiding operator for irreducible Koopman representations of unimodular groups with a $c$-tempered Folner sequence defined by a $K$-bi-invariant length function.) Suppose a locally compact unimodular group $G$ acts by measurable transformations on a probability space $(X,\nu)$. We assume that the action preserves the class of $\nu$.
For $g\in G$, we denote
$$x\mapsto c(g,x)=\sqrt{\frac{d(g^{-1}_*\nu)}{d\nu}(x)}$$ the cocycle defined by the square root of the Radon-Nikodym derivative of the action of $g^{-1}$.
We assume the associated Koopman representation
$$\pi:G\to U(L^2(X,\nu))$$
$$(\pi(g)\varphi)(x)=\varphi(g^{-1}x)c(g^{-1},x)$$
is strongly continuous. 
Let 
$$\Xi(g)=\langle\pi(g)\mathbf{1}_X,\mathbf{1}_X\rangle=\int_X c(g^{-1},x)d\nu(x).$$
Suppose a subgroup $K<G$ acts transitively on $X$. Suppose $K$ fixes $\mathbf{1}_X$ (in other words, $K$ preserves $\nu$).
Let $$L:G\to[0,\infty[$$ be a length function on $G$ which is $K$-bi-invariant.  Assume for each $n_0>0$
\[
\lim_{n\to\infty}\frac{\int_{B_{n+n_0}\setminus B_{n-n_0}}\Xi(g)^2dg}{\int_{B_n}\Xi(g)^2dg}=0.	
\]
Let $$d\mu_n(g)=\frac{\mathbf{1}_{B_n}(g)}{\int_{B_n}\Xi(g)^2dg}dg.$$ If $\pi$ is irreducible then
	\[
	\lim_{n\to\infty}\int_{G}\pi(g)\otimes{\pi}(g^{-1})d\mu_n(g)=F.	
	\]
\end{coro}

\begin{proof} It is enough to check that Property $(1)$ in Proposition \ref{prop:1,2,3} is verified. The uniform bound on the operator norm is true: for any $n>0$, the measure $\mu_n$ is symmetric and $K$-bi-invariant, hence Theorem \ref{theo:uniform}  implies that
	\[
	\|\int_{G}\pi(g)\otimes{\pi}(g^{-1})d\mu_n(g)\|_{op}=\int_{G}\Xi(g)^2d\mu_n(g),
	\]
and, by definition of $\mu_n$,
\[
\int_{G}\Xi(g)^2d\mu_n(g)=1.	
\]
In order to establish Schur's asymptotic relations, requested in Property $(1)$ from Proposition \ref{prop:1,2,3}, it is enough to check that the hypothesis of Proposition \ref{Identification of the limit} are fulfilled. To start with, $\pi$ is irreducible, by assumption. Next, a closed ball in $B(L^2(X,\nu)\otimes L^2(X,\nu))$ is compact in the weak operator topology. Hence, in order to show that the 
sequence
\[
	M_n=\int_{G}\pi(g)\otimes{\pi}(g^{-1})d\mu_n(g)
\]
converges to $F$, it is enough to show that any of its convergent subsequence converges to $F$.
So we may assume that $M_n$ converges to some operator $M\in B(L^2(X,\nu)\otimes L^2(X,\nu))$ and we have to prove that $M=F$. We check the remaining  hypothesis of Proposition \ref{Identification of the limit}: as generating set, we choose the group itself $S=G$, and as $\pi\otimes\pi$-invariant dense domain, we choose $L^2(X,\nu)\otimes_{\emph{alg}}L^2(X,\nu)$. Let $s,t\in G$.
In order to finish the proof, it is enough to show that
\[
\int_G\pi(g)\otimes\pi(g^{-1})d\mu_n(t^{-1}gs^{-1})-\int_G\pi(g)\otimes\pi(g^{-1})d\mu_n(g) 
\]
weakly converges to zero on $L^2(X,\nu)\otimes_{\emph{alg}}L^2(X,\nu)$. And, according to Example \ref{Folner c-temperness}, this convergence holds, as soon as for any $v,w\in L^2(X,\nu)$,
\[
\lim_{n\to\infty}\frac{1}{\int_{B_n}\Xi(g)^2dg}\int_{(sB_nt)\setminus B_n}|\langle\pi(g)v,w\rangle|^2dg=0.
\]		
In order to check that this limit is indeed zero, let $n_0=2\max\{L(s);L(t)\}$. We have: 
\[
(sB_nt)\setminus B_n\subset B_{n+n_0}\setminus B_{n-n_0}.		
\]
Hence:
\begin{align*}
\int_{(sB_nt)\setminus B_n}|\langle\pi(g)v,w\rangle|^2dg&\leq\int_{B_{n+n_0}\setminus B_{n-n_0}}|\langle\pi(g)v,w\rangle|^2dg\\
            &=\langle\int_{B_{n+n_0}\setminus B_{n-n_0}}\pi(g)\otimes\pi(g^{-1})dg(w\otimes v),v\otimes w\rangle\\ 
			&\leq\|\int_{B_{n+n_0}\setminus B_{n-n_0}}\pi(g)\otimes\pi(g^{-1})dg\|_{op}\|v\otimes w\|^2\\
			&=\int_{B_{n+n_0}\setminus B_{n-n_0}}\Xi(g)^2dg\,\|v\otimes w\|^2,\\
			\end{align*}
where the last equality follows from Theorem \ref{theo:uniform} applied to the symmetric $K$-bi-invariant measure
\[
	\mathbf{1}_{B_{n+n_0}\setminus B_{n-n_0}}(g)dg.
\]
\end{proof}

\section{Examples}

\subsection{Automorphism groups of regular trees}
Our aim is to apply Corollary \ref{coro} to the automorphism group of a regular tree and its associated boundary representation.
Let $(T,d)$ be the regular tree of degree $q+1$ equipped with its geodesic path metric $d$ for which each edge is isometric to the unit interval $[0,1]\subset\mathbb R$. Let $x_0$ be a vertex of $T$. Let $\partial T$ be its boundary at infinity (see \cite[Appendix]{PLP} for more details and references).  Let $b\in \partial T$ and let $\beta:[0,\infty)\to T$ be a geodesic ray representing $b$. Let $x, y\in T$. Let
\[
	B_b(x,y)=\lim_{t\to\infty}[d(x,\beta(t))-d(y,\beta(t))],
\]
be the Busemann cocycle defined by $b\in\partial T$.
Let $a, b\in\partial T$ and let $\alpha,\beta:[0,\infty)\to T$ be geodesic rays representing $a$ and $b$. Their Gromov product relative to the base point $x_0$ is defined as
\[
	(a|b)_{x_0}=\frac{1}{2}\lim_{t\to\infty}[d(x_0,\alpha(t))+d(x_0,\beta(t))-d(\alpha(t),\beta(t))].
\]
The formula
\[
	d_{x_0}(a,b)=e^{-(a|b)_{x_0}}
\]
defines an ultra-metric on $\partial T$. The group $\mbox{Aut}(T)$ of isometries of $T$ is locally compact unimodular and acts on $\partial T$ by conformal transformations.
The Hausdorff dimension of $(\partial T,d_{x_0})$ equals $\log q$ and the normalized Hausdorff measure $\nu$ on $(\partial T,d_{x_0})$ is the unique regular Borel
probability measure on $\partial T$ invariant under the action of the stabilizer $K=\mbox{Aut}(T)_{x_0}$ of $x_0$. The Radon-Nikodym derivative of $g\in \mbox{Aut}(T)$
at $b\in\partial T$ is
\[
	\frac{dg_*\nu}{d\nu}(b)=q^{B_b(x_0,gx_0)}.
\]
We denote   $\pi:\mbox{Aut}(T)\to U(L^2(\partial T,\nu))$ the corresponding Koopman representation.
Let $\mathbf{1}_{\partial T}\in L^2(\partial T,\nu)$ be the constant function equal to $1$. 
The Harish-Chandra function 
\[
	\Xi:\mbox{Aut}(T)\to(0,\infty)
\]
is the coefficient of $\pi$ defined by $\mathbf{1}_{\partial T}$ that is:
\[
\Xi(g)=\langle\pi(g)\mathbf{1}_{\partial T},\mathbf{1}_{\partial T}\rangle.
\]
As the action of $K$ preserves the measure and as $\pi$ is unitary, the Harish-Chandra function is $K$-bi-invariant and symmetric.
The length function
\[
	L:\mbox{Aut}(T)\to\mathbb N\cup\{0\}
\]
\[
	L(g)=d(x_0,gx_0)
\]
is also $K$-bi-invariant and symmetric (notice that the elements of length $0$ are the elements of $K$). As $K$ acts transitively on each sphere of $T$ with center $x_0$, if $g,g'\in \mbox{Aut}(T)$ satisfy $L(g)=L(g')$ then there exist $k,k'\in K$ such that $kg=g'k'$. This implies that $\Xi$ is constant on the level sets of $L$. For each $n\in\mathbb N\cup\{0\}$,
we will write $\Xi(n)$ for the common value of the Harish-Chandra function  on all $g\in \mbox{Aut}(T)$ such that $L(g)=n$. 
For any $n\in\mathbb N\cup\{0\}$,
\begin{align}\label{formula: Harish-Chandra}
\Xi(n)=\left(1+\frac{q-1}{q+1}n\right)q^{-n/2}.
\end{align}
Formula \ref{formula: Harish-Chandra}  implies that
\[
\int_{B_n}\Xi(g)^2dg\sim\frac{1}{3}\frac{(q-1)^2}{q(q+1)}n^3.	
\]
Hence the hypothesis of Corollary \ref{coro} on the behavior of $\Xi(g)$ is fulfilled. As $\pi$ is irreducible (see for example \cite{Figa} or \cite{BoyPin}), Corollary \ref{coro} implies that
\[
	\lim_{n\to\infty}\frac{1}{\int_{B_n}\Xi(g)^2dg}\int_{B_n}\pi(g)\otimes\pi(g^{-1})dg=F.
\]
Or:
\[
	\lim_{n\to\infty}\frac{1}{n^3}\int_{B_n}\pi(g)\otimes\pi(g^{-1})dg=\frac{1}{3}\frac{(q-1)^2}{q(q+1)}F.
\]
Applying Formulae (\ref{equation tau}) and (\ref{equation Schur}), we deduce that 
\[
	\lim_{n\to\infty}\frac{1}{n^3}\int_{B_n}\langle\pi(g)v,w\rangle\overline{\langle\pi(g)v',w'\rangle}dg=\frac{1}{3}\frac{(q-1)^2}{q(q+1)}\langle v,v'\rangle\overline{\langle w,w'\rangle},	
	\]
for any $v,w,v',w'\in L^2(\partial T,\nu)$.	
\subsection{Semisimple groups}
Let $G$ be a semisimple group over a local field. We assume that $G$ is non compact with finite center. Let  $P<G$ be a minimal parabolic subgroup and $K<G$ be a maximal compact subgroup such that $G=KP$. Let $\nu$ be the unique regular Borel probability measure on
$G/P=K/(K\cap P)$ which is $K$-invariant. Let $\pi:G\to U(L^2(G/P,\nu))$ be the quasi-regular representation on $\mathcal H=L^2(G/P,\nu)$ (see for example \cite[Example E.1.8 (ii), Definition B.1.9, Section A.6]{BDV}).
Let $(X,d)$ be either the Riemannian symmetric space with its usual $G$-invariant metric $d$ defined by the Killing form, or - in the case the local field is non-Archimedean - the Bruhat-Tits building associated to $G$ with a $G$-invariant metric $d$ making the apartments totally geodesic and isometric to Euclidean spaces.  Let $x_0\in X$ and  denote $L(g)=d(x_0,gx_0)$ the associated length function on $G$. All the hypothesis of Corollary \ref{coro} are obvious to check or well-known, except the condition on the growth of the Harish-Chandra function, which is the subject of \cite[Lemma 6.1]{KYD} (the length function $L$ corresponds to the radius function $r$ from \cite[1.3]{KYD}, see \cite[2.3 Décomposition de Cartan]{Qui} for details and references in the non-Archimedean case). We deduce that
\[
	\lim_{n\to\infty}\frac{1}{\int_{B_n}\Xi(g)^2dg}\int_{B_n}\pi(g)\otimes\pi(g^{-1})dg=F.
\]
Notice that \cite[Theorem 1.7]{KYD} implies that there exists an integer $d$ and a strictly positive number $f$ such that
\[
\int_{B_n}\Xi(g)^2dg\sim \frac{n^d}{f}.	
\]
Hence we also have:
\[
	\lim_{n\to\infty}\frac{1}{n^d}\int_{B_n}\pi(g)\otimes\pi(g^{-1})dg=\frac{1}{f}F.
\]
\section{Appendix}

\begin{lemma}\label{closure 1} Let $\mathcal H$ be a complex Hilbert space. Let $D<\mathcal H$ be a dense subspace and $M:D\to\mathcal H$ be a closable operator.
	Let $U:\mathcal H\to\mathcal H$ be a unitary transformation. Then $UM$ is closable and its closure is the closure of $M$ post-composed with $U$:
	\[
		\overline{UM}=U\overline{M}.
	\]
\end{lemma}
\begin{proof} We consider the isomorphism of Hilbert spaces
	\[
		\mathcal H\times\mathcal H\xrightarrow{I\times U}\mathcal H\times\mathcal H
	\]
	\[
		(x,y)\mapsto(x,U(y)).
	\]
The graph of $M$ is obviously sent to the graph of $UM$:
\[
	(I\times U)\Gamma(M)=\Gamma(UM).
\]
As $M$ is closable, the closure of its graph is the graph of its closure:
\[
\overline{\Gamma(M)}=\Gamma(\overline{M}).	
\]	
As $I\times U$ is an homeomorphism of $\mathcal H\times\mathcal H$,
\begin{align*}
	\overline{\Gamma(UM)}&=\overline{(I\times U)\Gamma(M)}=(I\times U)\overline{\Gamma(M)}\\
						 &=(I\times U)\Gamma(\overline{M})=\Gamma(U\overline{M}).
\end{align*}
This proves that $UM$ is closable with closure $U\overline{M}$.	
\end{proof}

\begin{lemma}\label{closure 2} Let $\mathcal H$ be a complex Hilbert space. Let $D<\mathcal H$ be a dense subspace and $M:D\to\mathcal H$ be a closable operator.
	Let $U,V:\mathcal H\to\mathcal H$ be unitary transformations. Assume that $V$ preserves $D$ and $UMV=M$. Then $V$ preserves the domain of the closure of $M$, i.e. $VD(\overline{M})<D(\overline{M})$ and $U\overline{M}V=\overline{M}$.
\end{lemma}
\begin{proof} We consider the isomorphism of Hilbert spaces
	\[
		\mathcal H\times\mathcal H\xrightarrow{V^{-1}\times U}\mathcal H\times\mathcal H
	\]
	\[
		(x,y)\mapsto(V^{-1}(x),U(y)).
	\]
It transforms the graph of $M$ in the graph of $UMV$:
\[
(V^{-1}\times U)\Gamma(M)=\Gamma(UMV).	
\]
As it is a homeomorphism, we obtain
\begin{align*}
\Gamma(\overline{M})&=\overline{\Gamma(M)}=\overline{\Gamma(UMV)}\\
					&=\overline{(V^{-1}\times U)\Gamma(M)}=(V^{-1}\times U)\overline{\Gamma(M)}\\
					&=(V^{-1}\times U)\Gamma(\overline{M}).
\end{align*}
Let $p_1:\mathcal H\times\mathcal H\to\mathcal H$ be the projection onto the first factor. For any $(x,y)\in\mathcal H\times\mathcal H$, we have:
\[
	V^{-1}p_1(x,y)=p_1(V^{-1}\times U)(x,y).
\] 
Hence the domain $D(\overline{M})$ of the closure of $M$ satisfies
\begin{align*}
D(\overline{M})&=p_1(\Gamma(\overline{M}))=p_1((V^{-1}\times U)\Gamma(\overline{M}))\\
				&=V^{-1}p_1(\Gamma(\overline{M}))=V^{-1}(D(\overline{M})).
\end{align*}
This proves that $V$ preserves $D(\overline{M})$. It is obvious that $\Gamma(U\overline{M}V)=(V^{-1}\times U)\Gamma(\overline{M})$. Hence,
\begin{align*}
\Gamma(\overline{M})=(V^{-1}\times U)\Gamma(\overline{M})=\Gamma(U\overline{M}V).	
\end{align*}
\end{proof}

\end{document}